\documentclass[11pt, reqno]{amsart}
\usepackage{amsmath}
\usepackage{amsfonts}
\usepackage{amssymb}
\usepackage{amscd}
\usepackage{amsthm}
\usepackage[mathscr]{eucal}
\theoremstyle{plain}
\theoremstyle{definition}
\newtheorem{theorem}{Theorem}[section]
\newtheorem{lemma}[theorem]{Lemma}
\newtheorem{proposition}[theorem]{Proposition}
\newtheorem{corollary}[theorem]{Corollary}
\newtheorem{definition}[theorem]{Definition}
\newtheorem{remark}[theorem]{Remark}
\newtheorem*{nnremark}{Remark}

\newtheorem*{claim}{Claim}
\newcommand{\mc}[1]{{\mathcal #1}}

\newcommand{\seq}[2]{\ensuremath{({#1}_{#2})_{#2\in\mathbb{N}}}}

\DeclareMathOperator{\essinf}{essinf}
\DeclareMathOperator{\Av}{Av} 

\newcommand{\mb}[1]{\mathbb{#1}}
\newcommand{\R}{\mathbb{R}}
\newcommand{\N}{\mathbb{N}}

\newcommand{\mtf}{\mc{M}_{\mc{T}}\phi}
\newcommand{\e}{\varepsilon}

\theoremstyle{plain}
 
\begin{document}
\title[sharp  $L^1$ estimates for the dyadic maximal operator]{An alternative approach\\to sharp  $L^1$ estimates\\for the dyadic maximal operator}
\begin{abstract}
We provide  alternative proofs of sharp  $L^1$ inequalities for the dyadic maximal function $\mc{M}_{\mc{T}}\phi$ when $\phi$ satisfies certain
$L^1$ and $L^{\infty}$ conditions (see \cite{M2}).
 \end{abstract}
 \author{Eleftherios  N. Nikolidakis}
\address[E. Nikolidakis]{Department of Mathematics,
University of Ioannina}
\email{enikolid@uoi.gr}
\author{Andreas G. Tolias}
\address[A. Tolias]{Department of  Mathematics,
University of Ioannina} \email{atolias@uoi.gr}
\maketitle
\section{Introduction} \label{S1}
The dyadic maximal operator on $\R^n$ is a useful tool in analysis and is defined by the formula
\begin{equation} \label{eq:1p1}
\mc{M}_d\phi(x) = \sup\left\{ \frac{1}{|S|} \int_S |\phi(u)| du: x\in S,\; S\subset  \R^n\ \text{ is a dyadic cube} \right\},
\end{equation}
for every $\phi\in L^1_\text{loc}(\mb R^n)$, where $|\cdot|$ denotes the Lebesgue measure on $\mb R^n$, and the dyadic cubes are those formed by the grids $2^{-N}\mb Z^n$, for $N=0, 1, 2, \ldots$.\\
It is well known that the operator defined above satisfies the following weak type $(1,1)$ inequality
\begin{equation} \label{eq:1p2}
\left|\left\{ x\in\mb R^n: \mc M_d\phi(x) > \lambda \right\}\right| \leq \frac{1}{\lambda} \int_{\left\{\mc M_d\phi > \lambda\right\}} |\phi(u)|\,  du
\end{equation}
for every $\phi\in L^1(\mb R^n)$ and every $\lambda>0$,
from which it is easy to get the following  $L^p$-inequality
\begin{equation} \label{eq:1p3}
\|\mc M_d\phi\|_p \leq \frac{p}{p-1} \|\phi\|_p
\end{equation}
for every $p>1$, and every $\phi\in L^p(\mb R^n)$.
It is easy to see that the weak type inequality \eqref{eq:1p2} is the best possible. For refinements of this inequality see \cite{N2}.

It has also been proved that \eqref{eq:1p3} is best possible (see \cite{B1} and \cite{B2} for general martingales and \cite{Wa} for dyadic ones).
An approach for the study of the behaviour of this maximal operator in more depth is the introduction of the so called Bellman functions which play the role
 of generalized norms of $\mc M_d$. Such functions related to the $L^p$-inequality \eqref{eq:1p3} have been precisely evaluated in   \cite{M1},
  \cite{M2} and \cite{NM1}. For the study of the Bellman functions of $\mc M_d$, we use the notation $\Av_E(\psi)=\frac{1}{|E|} \int_E \psi$,
   whenever $E$ is a Lebesgue measurable subset of $\mb R^n$ of positive measure and $\psi$ is a real valued integrable function
    defined on $E$. For a fixed  dyadic cube  $Q$  the localized maximal operator $\mc M'_d\phi$ is defined as in \eqref{eq:1p1}
    but with the dyadic cubes $S$ being assumed to be contained in $Q$. Then for every $p>1$  let
\begin{equation} \label{eq:1p4}
B_p(f,F)=\sup\left\{ \frac{1}{|Q|} \int_Q (\mc M'_d\phi)^p: \; \phi\ge 0,\; \Av_Q(\phi)=f,\; \Av_Q(\phi^p)=F \right\}
\end{equation}
where   the variables $f, F$ satisfy $0<f^p\leq F$.
This is the well known Bellman function  of two integral variables of the dyadic maximal operator.
By a scaling argument it is easy to see that \eqref{eq:1p4} is independent of the choice of $Q$, so we may choose
$Q$ to be the unit cube $[0,1]^n$.
In \cite{M1}, the function \eqref{eq:1p4} has been precisely evaluated for the first time. The proof has been given in a much more general setting of tree-like structures on probability spaces.

More precisely for a non-atomic probability space $(X,\mc{A},\mu)$ and  $\mc T$  a family of measurable subsets of $X$ that has a tree-like structure similar to the one of the dyadic case (the exact definition is given in Section \ref{S2})
  the dyadic maximal operator associated to $\mc T$ is defined  by
\begin{equation} \label{eq:1p5}
\mc M_{\mc T}\phi(x) = \sup \left\{ \frac{1}{\mu(I)} \int_I |\phi| d\mu:\; x\in I\in \mc T \right\}
\end{equation}
for every $\phi\in L^1(\mu)$  and $x\in X$.

This operator is related to the theory of martingales and satisfies essentially the same inequalities as $\mc M_d$ does. Now we define the corresponding Bellman function of three  variables of $\mc M_{\mc T}$, by
\begin{multline} \label{eq:1p6}
B_p^{\mc T}(f,F,k) = \sup \left\{ \int_K (\mc M_{\mc T}\phi)^p  d\mu:\; \phi\geq 0,\; \int_X\phi d\mu=f, \right. \\
  \left. \int_X\phi^p d\mu = F,\ K\subset X\ \text{measurable with}\ \mu(K)=k\right\},
\end{multline}
the variables $f, F,  k$ satisfying $0<f^p\leq F $ and   $k\in (0,1]$.
The exact evaluation of \eqref{eq:1p6} is given in \cite{M1}.

It is well known that in general  $\mc M_{\mc T}\phi$ does not belong to $L^1(\mu)$ when $\phi\in L^1(\mu)$. In \cite{Wi} it is proved
that if $\phi$ satisfies the condition $\int\limits_X |\phi| \log^+|\phi| d\mu  <+\infty $ then $\mc M_{\mc T}\phi\in L^1(\mu)$.
 In \cite{S} it is shown that this condition is also necessary for the integrability  of $\mc M_{\mc T}\phi$.
In \cite{M3} the corresponding to \eqref{eq:1p6} function with respect to certain $L\log L$ conditions has been precisely evaluated.
As a  matter of fact in \cite{M3} more general conditions on   $\phi$ have been considered.
An application of this result is the evaluation of the following Bellman type function
\begin{equation} \label{eq:1p7}
B_1^{\mc T}(f,M) = \sup \left\{ \int_X \mc M_{\mc T}\phi    d\mu:\; \phi\geq 0,\; \int_X\phi d\mu=f, \;
 \|\phi\|_\infty=M\right\}
\end{equation}
when $0<f\le M$.

In the subsequent sections we provide proofs of Theorem \ref{T1} and Theorem \ref{T2} that are stated right below. The results that we present are
special cases of deep results concerning the study of more general Bellman type functions that are considered in \cite{M2}  by A. Melas. However the approach that
we give in the present paper is more simple and elementary and thus easily accessible to the reader.

\begin{theorem}\label{T1}
For all real variables $f,M_1,M_2$ with $M_1\ge f>M_2\ge 0$ the following holds:
\begin{multline}\label{teq8}
\sup\Big\{ \int\limits_X  \mc{M}_{\mc{T}}\phi d\mu:     \phi:X\to \R^+ \mbox{ is measurable, }\\[1mm]
    \int\limits_X \phi d\mu=f,\; \|\phi\|_{\infty}=M_1,\;
   \essinf_X(\phi)=M_2   \Big\}\\[1mm]
   =  f+(f-M_2)\log\Big(\frac{M_1-M_2}{f-M_2}\Big)
 \end{multline}
 \end{theorem}

For the proof of the above theorem we study the respective Hardy operator problem which is connected to the dyadic
maximal operator problem and we use a symmetrization principle which appears in \cite{NM1}.

\begin{theorem}\label{T2}
For all $f,M,k$ that satisfy $0<f\le M$ and $k\in(0,1]$  it holds that
\begin{multline}\label{ineq1}
 \sup\Big\{ \int\limits_K  \mc{M}_{\mc{T}}\phi d\mu:\;\phi:X\to \R^+ \mbox{ is measurable},\;
  \int\limits_X \phi d\mu=f,\;\|\phi\|_{\infty}=M,  \\[1mm]
  \; K \mbox{ measurable, }    \mu(K)=k      \Big\}                   \\[1mm]
=  \left\{ \begin{array} {l@{\quad} l}
        kM    & \mbox{ if }0<k\le \frac{f}{M}  \\[4mm]
   f+f\log(\frac{Mk}{f}) & \mbox{ if } \frac{f}{M}<k\le 1
   \end{array}\right.
  \end{multline}
\end{theorem}

   The values of the  supremums
 that appear in Theorem \ref{T1}  and  Theorem \ref{T2} are
   independent
  of the probability measure space $(X,\mc{A},\mu)$  and the tree $\mc{T}$ (see also \cite{M2}).
  
  At this points we should comment on the methods that we use in the proofs of Theorems  \ref{T1} and \ref{T2}  compared to
  the methods used in the proofs of more general results in \cite{M2} by A. Melas.
  
In \cite{M2} A. Melas studies a more general problem by considering integrals of $\phi$ and $\mc{M}_{\mc{T}}\phi$ related to two increasing
convex functions $G$ and $H$ that satisfy certain growth conditions. His approach is given in several steps.
In the first one he provides a combinatorial rearrangement inequality on subtrees of the initial tree $\mc{T}$, and several technical lemmas that
uses in the sequel.
 In the second step he applies a linearization for $\mc{M}_{\mc{T}}\phi$ which permits him to study this maximal function on a certain subtree
 of $\mc{T}$ related to $\phi$.
By using the rearrangement inequality proved in the first step he reduces the evaluation of the Bellman function of interest to the evaluation
of a respective  Bellman type function involving decreasing functions.
The proofs of his results involve techniques from ODE's and from the theory of calculus of variations. At the third step he finds extremals
for the Bellman functions that he studies, again using ODE's and several techniques on extremization of integral expressions.
Finally he provides   examples, considering specific functions $G$ and $H$.

In our approach we use independent results (appearing in \cite{NM1}  and \cite{N3})  that allow us to reduce the evaluation
of the Bellman type functions that we study, to the corresponding problem for the Hardy operator acting on decreasing functions on $(0,1]$.
We determine the upper bound that is described in \eqref{teq8} (see Theorem \ref{T1}) by using Riemann-Stieljes integrals in a direct way.
 The sharpness of this upper bound in \eqref{teq8} is proved by using the results in \cite{N3}. Then we prove Theorem \ref{T2} by considering integrals of $\mc{M}_{\mc{T}}\phi$ on certain subsets of $X$ (related to the distribution function of $\mc{M}_{\mc{T}}\phi$)
 which can be decomposed as pairwise almost  disjoint unions of elements of $\mc{T}$. Then applying Theorem \ref{T1} we reach the upper
 bound that is stated in   \eqref{ineq1}.  Finally we prove the sharpness of this upper bound by providing functions $\phi$
 which satisfy the conditions that are settled on our problem, for which the value of $\int\limits_K \mc{M}_{\mc{T}}\phi d \mu$ on certain
 suitable subsets $K$ of $X$ is arbitrarily close to the right side of \eqref{ineq1}.

\section{Preliminaries} \label{S2}
\begin{definition}
Let  $(X,\mc{A},\mu)$ be a non-atomic probability measure space.
We recall that a  collection of  measurable sets $\mc{T}$ is called a tree in $\mc{A}$ provided that the following
 conditions are satisfied:
 \begin{enumerate}
 \item[(i)]    $X\in\mc{T}$ and every $I\in \mc{T}$ has positive measure.
  \item[(ii)]  To every $I\in \mc{T}$ corresponds a countable (finite or infinite) family $C(I)\subset \mc{T}$, containing at least two elements,
 such that:
 \begin{enumerate}
  \item[(a)]  The elements of $C(I)$ are almost pairwise disjoint, i.e. for $J,J'\in C(I)$ with $J\neq J'$ we have that
   $\mu(J\cap J')=0$.
  \item[(b)]   $I=\cup  C(I)$.
 \end{enumerate}
   \item[(iii)]  If we define $\mc{T}_{(0)}=\{X\}$ and $\mc{T}_{(n+1)}=\cup \{C(I): \;  I\in \mc{T}_{(n)}\}$ for all $n$
    then $\mc{T}=\bigcup\limits_{n=0}^{\infty} \mc{T}_{(n)}$.
  \item[(iv)] For     $(\mc{T}_{(n)})_{n\in\N}$ as defined above,
     $\lim\limits_{n}  \big[  \sup\{\mu(I):  I\in  \mc{T}_{(n)}\}\big]=0$.
 \end{enumerate}
 \end{definition}
 The maximal operator $\mc{M}_{\mc{T}}$ associated to the tree $\mc{T}$ corresponds to every measurable function
 $\phi:X\to\R$
  the  function $\mc{M}_\mc{T} \phi$ defined by the formula
\[ \mc{M}_\mc{T} \phi(x)=\sup  \big\{\frac{1}{\mu(I)}\int\limits_I|\phi| d \mu:   \;\; x\in I\in\mc{T}     \big\}.\]

We also recall that for every measurable function $\phi:X\to \R$,
  defining
   $\phi^*:(0,1]\to [0,+\infty)$   by the formula
   \[    \phi^*(t)=\inf\{y>0:\;  \mu([|\phi|>y])<t\}  \]
   we have that $\phi^*$ is the unique decreasing and left continuous function on $(0,1]$ that is equimeasurable to $\phi$.
 The following is proved in \cite{M1}:

 \begin{lemma}\label{L2}
    For every $I\in\mc{T}$ and every $\alpha$ such that $0<\alpha<1$, there exists a subfamily $\mc{F}$ of $\mc{T}$
  consisting of almost pairwise disjoint subsets of $I$  such that
    \[ \mu\big(\bigcup\limits_{J\in \mc{F}}J\big)=\sum\limits_{J\in \mc{F}}\mu(J)= (1-\alpha)\mu(I).  \]
  \end{lemma}

  In \cite{N3}  the following is proved:

 \begin{lemma}
   For any integrable $\phi:X\to \R^+$ it holds that
   \[  (\mc{M}_{\mc{T}}f)^*(t)\le \frac{1}{t}\int_0^t\phi^*(u)du\mbox{  for every }t\in(0,1].\]
  \end{lemma}

 We will also need the following symmetrization principle which appears in \cite{NM1}.

 \begin{theorem}\label{T3}
 Let $g:(0,1]\to\R^+$ be decreasing  and let  $G_1, G_2:[0,+\infty) \to [0,+\infty) $ be two increasing functions.
 Then for every $k\in (0,1]$ the following holds:
 \begin{multline*}
  \sup\big\{\int_K (G_1\circ \mc{M}_{\mc{T}}\phi)(G_2\circ \phi)d\mu:\;  \phi^*=g,\; \mu(K)=k\big\}\\[1mm]
  =\int_0^k G_1\big(\frac{1}{t}\int_0^t g(u)du\big)G_2\big(g(t)\big)dt.
  \end{multline*}
 \end{theorem}

 Let us fix some notation. For   $M_1\ge f  > M_2\ge 0$, we set
 \begin{eqnarray*}
  \mc{C}_{X,\mc{T}}(M_1,f,M_2) & =  &
     \{ \phi:X\to \R^+ \text{ measurable, } \|\phi\|_{\infty} = M_1,  \\
     & &    \;\int\limits_X \phi d\mu=f,\;  \essinf_X(\phi)=M_2  \}
   \end{eqnarray*}
   and
   \begin{equation*}
   \mc{C}_{X,\mc{T}}(M_1,f)  =
     \{ \phi:X\to \R^+ \text{ measurable, } \|\phi\|_{\infty} = M_1,
    \;\int\limits_X \phi d\mu=f  \}.
     \end{equation*}

 In Theorem \ref{T1}  we will compute the quantity
\[ \sup\big\{    \int\limits_{X} \mc{M}_{\mc{T}}\phi d \mu:\;    \phi\in   \mc{C}_{X,\mc{T}}(M_1,f,M_2)  \big\}  \]
for all
  for all $M_1\ge f >M_2\ge 0$  while in Theorem \ref{T2} we will compute   the quantity
\[     \sup\big\{ \int\limits_{K} \mc{M}_{\mc{T}}\phi d \mu, \;\phi\in   \mc{C}_{X,\mc{T}}(M_1,f),  \; K\in \mc{A},\;\mu(K)=k   \big\}  \]
    for all $M_1\ge f >0$ and all $k\in(0,1]$.

 \section{ Proof of Theorem \ref{T1}} \label{S3}
For the proof of Theorem \ref{T1} we work as follows.
Fix  $M_1\ge f>M_2\ge 0$
and let $\mc{A}(M_1,f,M_2)$ be the class of functions
\begin{eqnarray*}
\mc{A}(M_1,f,M_2)&=& \Big\{g:[0,1)\to [0,+\infty):   \;\; g  \mbox{ is decreasing,}\\
 &    &    \mbox{left continuous, continuous at }  0,\\
      &   & g(0)=M_1,\; \lim\limits_{t\to 1}g(t)=M_2 \mbox{ and }
 \int_0^1g(t)dt=f \Big\}
\end{eqnarray*}

 For each $g\in\mc{A}(M_1,f,M_2)$ we set
\[I_g=\int_0^1\Big(\frac{1}{t}\int_0^tg(u)du\Big)dt.\]

 Our goal is to maximize $I_g$  over all $g\in\mc{A}(M_1,f,M_2)$.

We observe that since $\lim\limits_{t\to 0} \Big(\log(t)\int_0^t g(u)du\Big)=0$ for every \\ $g\in\mc{A}(M_1,f,M_2)$,
 when we integrate by parts we get that
\[
 I_g=\int_0^1\Big(\frac{1}{t}\int_0^t g(u)du\Big)dt=-\int_0^1\log(t)g(t)dt.\]
 We consider the function $h:[0,1]\to \R$ defined by the formula
 \[h(t)=    \left\{ \begin{array} {l@{\quad} l}
      t-t\log(t) & \mbox{ if }\;\;0<t\le 1 \\[4mm]
  0 & \mbox{ if }\;\;t=0.     \end{array}\right.
     \]
 Since  $h(0)=0=\lim\limits_{t\to 0}h(t)$, $h'(t)=-\log(t)>0$   and    $h''(t)=-\frac{1}{t}<0$  for every $t\in (0,1)$,
 $h$ is continuous, strictly  increasing and strictly concave, thus
  for every $g\in \mc{A}(M_1,f,M_2)$ the value $I_g=\int_0^1g(t)\big(-\log(t)\big)dt$ may be expressed as the Riemann-Stieltjes integral $I_g=\int_0^1g(t)dh(t).$

\begin{proposition}\label{pr1}
 The quantity  $\sup\{I_g:\; g\in\mc{A}(M_1,f,M_2)\}$ is equal to \\$f+(f-M_2)\log(\frac{M_1-M_2}{f-M_2})$.
Moreover the supremum  is    uniquely attained by the function $g_0$ defined as
 \[g_0(t)=    \left\{ \begin{array} {l@{\quad} l}
  M_1 & \mbox{ if }\;\;0\le t\le c \\[4mm]
  M_2 & \mbox{ if }\;\;c<t\le 1     \end{array}\right.
     \]
where $c=\frac{f-M_2}{M_1-M_2}$.
\end{proposition}
\begin{proof}[\bf Proof.]
We start with the second part.  A direct calculation shows that \[\int_0^1g_0(t)dt=M_1 \frac{f-M_2}{M_1-M_2}+M_2(1-\frac{f-M_2}{M_1-M_2})=f,\] while $g_0(0)=M_1$, $g_0(1)=M_2$ and $g_0$ is left continuous; thus $g_0\in\mc{A}(M_1,f,M_2)$.
We also have that
\begin{eqnarray*}
I_{g_0}  & =&  -\int_0^1\log(t)g_0(t)dt=\ -\int_0^cM_1\log(t)dt- \int_c^1M_2\log(t)dt\\
           & = & M_1 \big(h(c)-h(0)\big)+M_2 \big( h(1)-h(c)\big)\\
            &= &  M_1\big(c-c\log(c))+M_2\big(1-(c-c\log(c)\big)  \\
           &= &  M_2+(M_1-M_2) \; c\; \big(1-\log(c)\big)   \\
                 &= & M_2+(M_1-M_2)\frac{f-M_2}{M_1-M_2}\big(1-\log (\frac{f-M_2}{M_1-M_2})\big)\\
              &=  &   f+(f-M_2) \log (\frac{M_1-M_2}{f-M_2}).
\end{eqnarray*}
Thus $\sup\{I_g:\;  g\in \mc{A}(M_1,f,M_2)\}\ge   f+(f-M_2) \log (\frac{M_1-M_2}{f-M_2})$.

We will show now that the number  $  f+(f-M_2) \log (\frac{M_1-M_2}{f-M_2})$  is an upper
bound   of the set   $\{I_g:\; g\in  \mc{A}(M_1,f,M_2)\}$.

Let $g\in \mc{A}(M_1,f,M_2)$.
 We will calculate the Riemman Stieltjes integral $I_g=\int_0^1g(t)dh(t)$
 as the limit of the  Riemman Stieltjes  sums over the  net of all partitions $\mc{P}=\{0=t_0<t_1<t_2<\cdots<t_n=1\}$  equipped with   the right  boundaries of the intervals $[t_{i-1},t_i]$ as intermediate points.

 \begin{eqnarray*}
 I_g& = & \int_0^1g(t)dh(t)  \\
     &=&  \lim\limits_{\mc{P}}  \sum\limits_{i=1}^n g(t_i)\big(h(t_i)-h(t_{i-1})\big)\\
       &=&  \lim\limits_{\mc{P}}  \big(\sum\limits_{i=1}^n  g(t_i)h(t_i)  -\sum\limits_{i=1}^n  g(t_i)h(t_{i-1})\big)\\
             &=&  \lim\limits_{\mc{P}} \;\; \Big[ -g(t_1)h(t_0)+\sum\limits_{i=1}^{n-1}h(t_i)\big(g(t_i)-g(t_{i+1})\big)    +g(t_n)h(t_n)\Big]\\
            &=&  \lim\limits_{\mc{P}}   \;\; \Big[ M_2+ \sum\limits_{i=0}^{n-1}h(t_i)\big(g(t_i)-g(t_{i+1})\big)   \Big]
 \end{eqnarray*}
  (we took  into account that $h(t_0)=h(0)=0$, $h(t_n)=h(1)=1$ and $g(t_n)=g(1)=M_2$).

  Since $g$ is decreasing, for every partition $\mc{P}=\{0=t_0<t_1<t_2<\cdots<t_n=1\}$ we have that $g(t_i)-g(t_{i+1})\ge 0$
   for $i=0,1,\ldots,n-1$ while $\sum\limits_{i=0}^{n-1}\big(g(t_i)-g(t_{i+1})\big)=g(0)-g(1)=M_1-M_2$
 and thus the numbers $\big(\frac{g(t_i)-g(t_{i+1})}{M_1-M_2}\big)_{i=0}^{n-1}$ serve as coefficients of a convex combination. Since
   \[  I_g= M_2+(M_1-M_2)    \lim\limits_{\mc{P}}      \sum\limits_{i=0}^{n-1}h(t_i)           \frac{g(t_i)-g(t_{i+1})}{M_1-M_2}  \]
  the fact that the function $h$ is concave and continuous yields
   \begin{equation}\label{L1}
   I_g\le M_2+(M_1-M_2)  h\Big(        \lim\limits_{\mc{P}}     \sum\limits_{i=0}^{n-1}       \frac{g(t_i)-g(t_{i+1})}{M_1-M_2} t_i             \Big)
 \end{equation}
But also
 \begin{eqnarray*}
   & \lim\limits_{\mc{P}}     \sum\limits_{i=0}^{n-1}  \big(g(t_i)-g(t_{i+1})\big)t_i       &  \\
  &   =    \lim\limits_{\mc{P}}   \Big(   g(t_0)t_0+  \sum\limits_{i=0}^{n-1} g(t_i)(t_i-t_{i-1})- g(t_n)t_{n-1}\Big)  & \\
   &   =   M_1\cdot 0 +\int_0^1 g(t)dt -  M_2\cdot 1 =   f-M_2&
  \end{eqnarray*}
  and thus
\begin{eqnarray*}
  I_g  &  \le & M_2 +(M_1-M_2)h(\frac{f-M_2}{M_1-M_2})\\
       & =   &  M_2+ (M_1-M_2) \frac{f-M_2}{M_1-M_2}  (1-\log \frac{f-M_2}{M_1-M_2})\\
       &=  &  M_2+(f-M_2)+ (f-M_2)\log\frac{M_1-M_2}{f-M_2}\\
     &   =  & f+(f-M_2)\log \frac{M_1-M_2}{f-M_2}.
\end{eqnarray*}

Up to this pont we have shown that
 \[  \sup\{I_g: g\in \mc{A}(M_1,f,M_2)\}=  f+(f-  M_2)\log \frac{M_1-M_2}{f-M_2}\]  and that the
  supremum is   attained by the function $g_0$.  It remains to prove that $g_0$ is the unique function in  $\mc{A}(M_1,f,M_2)$ with this property.
 It is enough to show that for each $g\in  \mc{A}(M_1,f,M_2)$ with $g\neq g_0$ we have that $I_g< f+(f-  M_2)\log \frac{M_1-M_2}{f-M_2}$.

  Consider  such a $g$. Since $g\neq g_0$ there exists $y_0$ a point of continuity of $g$ such that $M_1>g(y_0)>M_2$. We set $M=g(y_0)$.
   In the proof that is presented above   we may consider only partitions of $[0,1]$ containing the point $y_0$  i.e. partitions  of the form
\[    \mc{P}=\{0=t_0<t_1<t_2<\cdots<t_{k-1}<t_k=y_0<t_{k+1}<\cdots<t_n=1 \}.    \]

 Then
\begin{eqnarray*}
  I_g  &=  & M_2 +(M_1-M_2) \lim\limits_{\mc{P}}\Big(\sum\limits_{i=0}^{n}h(t_i) \frac{g(t_i)-g(t_{i+1})}{M_1-M_2}\Big)  \\
              & =  &     M_2 +(M_1-M_2) \lim\limits_{\mc{P}} \Big(\frac{M_1-M}{M_1-M_2}   \sum\limits_{i=0}^{ k-1}h(t_i)\frac{g(t_i)-g(t_{i+1})}{M_1-M}\\
    &  &\qquad  + \frac{M-M_2}{M_1-M_2}    \sum\limits_{i=k}^{ n-1}h(t_i)\frac{g(t_i)-g(t_{i+1})}{M-M_2}
                                                  \Big)
\end{eqnarray*}
 Since $\sum\limits_{i=0}^{ k-1}\frac{g(t_i)-g(t_{i+1})}{M_1-M}=1$ and   $\sum\limits_{i=k}^{ n-1}\frac{g(t_i)-g(t_{i+1})}{M-M_2}=1$, using the fact that $h$ is concave we get that
\begin{eqnarray*}    I_g  & \le  &  M_2+(M_1-M_2)  \lim\limits_{\mc{P}}\Big(    \frac{M_1-M}{M_1-M_2}   h\big(     \sum\limits_{i=0}^{k-1} \frac{g(t_i)-g(t_{i+1})}{M_1-M}  t_i     \big)\\
&  &   \qquad       +                    \frac{M-M_2}{M_1-M_2}               h\big(     \sum\limits_{i=k}^{n-1} \frac{g(t_i)-g(t_{i+1})}{M-M_2}  t_i     \big)                \Big)
\end{eqnarray*}
But

\begin{eqnarray*}
\lim\limits_{\mc{P} }  \sum\limits_{i=0}^{k-1}\big(   g(t_i)-g(t_{i+1})\big) t_i     &=& \lim\limits_{\mc{P} }  \;\;\Big[g(t_0)t_0+\sum\limits_{i=1}^{k-1}g(t_i)(t_i-t_{i-1})-g(t_k)t_{k-1} \Big]  \\
    \\
   &= &  M_1   \cdot 0+\int_0^{y_0}g(t)dt-My_0=\int_0^{y_0}g(t)dt-My_0
\end{eqnarray*}
 and
\begin{eqnarray*}
   \lim\limits_{\mc{P} }  \sum\limits_{i=k}^{n-1}\big(   g(t_i)-g(t_{i+1})\big) t_i   &= &
   \lim\limits_{\mc{P} }  \;\;g(t_k)t_k+\sum\limits_{i=k+1}^{n-1}g(t_i)(t_i-t_{i-1})-g(t_n)t_{n-1} \\
     &=&
    M\cdot y_0  +\int_{y_0}^1g(t)dt-M_2.
\end{eqnarray*}
as $g$ is continuous at $y_0$ and $t_{k-1}\stackrel{\mc{P}}{\to} y_0$.

Thus we get that
\begin{eqnarray*}
I_g\le  M_2+(M_1-M_2)\Big[\frac{M_1-M}{M_1-M_2}h\Big(\frac{  \int_0^{y_0}g(t)dt-My_0   }{M_1-M}     \Big)\\
+    \frac{M-M_2}{M_1-M_2}   h\Big(  \frac{ My_0 +\int_{y_0}^1 g(t)dt -M_2   }{M-M_2}\Big)                                                            \Big]
\end{eqnarray*}

We  observe that the coefficients $a=\frac{M_1-M}{M_1-M_2}$ and $\beta=   \frac{M-M_2}{M_1-M_2}$ are positive, $a+\beta=1$ while,
\[        \frac{     \int_0^{y_0}g(t)dt-My_0}{M_1-M}  <y_0<   \frac{ My_0 +\int_{y_0}^1 g(t)dt -M_2   }{M-M_2}\]
where we took into account that the function $g$ is continuous at the points $y_0$ and $1$ and $M_1>g(y_0)=M>M_2$.

Therefore the fact that $h$ is strictly concave yields
\begin{eqnarray*}
   I_g   <   M_2+(M_1-M_2) h \Big( \frac{M_1-M}{M_1-M_2}\cdot\frac{  \int_0^{y_0}g(t)dt-My_0   }{M_1-M}   \\
+ \frac{M-M_2}{M_1-M_2}\cdot\frac{ My_0 +\int_{y_0}^1 g(t)dt -M_2   }{M-M_2}\Big)\\
=M_2+(M_1-M_2)h\Big( \frac{\int_0^1g(t)dt-M_2}{M_1-M_2}\Big)\\
= M_2+(M_1-M_2)h\big(\frac{f-M_2}{M_1-M_2}\big)\\
=f+(f-M_2)\log\big(\frac{M_1-M_2}{f-M_2}\big).
\end{eqnarray*}
The proof is complete.
\end{proof}

 Similar arguments to those used in the previous proof lead to the following.

 \begin{proposition} \label{pr2}
If   \seq{g}{n}  is a sequence in  $\mc{A}(M_1,f,M_2)$ such that
$\lim\limits_n   I_{g_n}=    f+(f-M_2)\log(\frac{M_1-M_2}{f-M_2})$ then
$g_n\stackrel{a.e.}{\longrightarrow}g_0$  (and therefore also $g_n\stackrel{L_1}{\longrightarrow}g_0$),
where $g_0$ is the function defined in the statement of Proposition \ref{pr1}.
 \end{proposition}

  We are now in position to prove Theorem \ref{T1}.

 \begin{proof}[\bf Proof of Theorem \ref{T1}]
 First we observe that for $\phi\in \mc{C}_{X,\mc{T}}(M_1,f,M_2)$ we have that $\phi^*\in \mc{A}(M_1,f,M_2)$.
 Also from \cite{N3} we have that
 $(\mc{M}_{\mc{T}}\phi)^*(t)\le \frac{1}{t}\int_0^t \phi^*(s)ds$ for every $t\in (0,1]$.
  Since the functions $\mc{M}_{\mc{T}}\phi$ and $(\mc{M}_{\mc{T}}\phi)^*$ are equimeasurable  we get that
  \begin{eqnarray*}
  \int\mc{M}_{\mc{T}}\phi d\mu & = & \int_0^1 (\mc{M}_{\mc{T}}\phi)^*(t)dt\\
    & \le &   \int_0^1 \Big(\frac{1}{t}\int_0^t \phi^*(s)ds\Big) dt \\
   &  = &  I_{\phi^*} \le  f+(f-M_2)\log(\frac{M_1-M_2}{f-M_2}).
      \end{eqnarray*}
where the last inequality follows from Proposition \ref{pr1}.

The sharpness of the above inequality is a consequence of Theorem     \ref{T3}
for $G_1(t)=t$, $G_2(t)=1$, $t\in [0,+\infty)$ and $k=1$,  where   $g=g_0$ is the function in
the statement of Proposition \ref{pr1} and the fact that since the probability measure space
$(X,\mc{A},\mu)$ is non-atomic we may easily find a measurable function $\phi:X\to\R$ such that $\phi^*=g_0$.
 The proof of Theorem \ref{T1} is complete.
 \end{proof}

 \section{Proof of Theorem \ref{T2}} \label{S4}

 In this section we will prove Theorem \ref{T2}. We start with the following lemma.
\begin{lemma}  \label{lem1}
Let $g_1,g_2:[0,+\infty)\to [0,+\infty)$ and $0<k\le 1$ such that:\\
(a) $g_1,g_2$ are decreasing.\\
(b) $g_1(t)\le  g_2(t)$ for every $t\in  [0,+\infty)$.\\
(c) $g_1$ is right continuous and $g_2$ is left continuous.\\
(d) $\lim\limits_{t\to+\infty}g_2(t)=0$ and $g_2(0)=1$.
Then there exists $u\in [0,+\infty)$ such that $g_1(u)\le k\le g_2(u)$.
\end{lemma}
\begin{proof}[\bf Proof.]
We set  $u= \sup\; g_2^{-1}\Big([k,+\infty)\Big)$.
Since $g_2(0)=1$ the set $g_2^{-1}\Big([k,+\infty)\Big)$ is nonempty, while it is bounded from above since $g_2$ is decreasing
 and  $\lim\limits_{t\to+\infty}g_2(t)=0$.  Thus $u$ is a well defined real number.

 The left continuity of $g_2$ yields $g_2(u)\ge k$, thus it suffices to show that $g_1(u)\le k$.
If $g_1(u)>k$ from the right continuity of $g_1$ we get that $g_1(t)>k$ for all $t\in [u,u+\e)$ for some $\e>0$.
Thus $k<g_1(t)\le g_2(t)$ for all $t\in [u,u+\e)$, which contradicts the definition of $u$.
\end{proof}

\begin{lemma}  \label{lem2}
Let $(X,\mc{A},\mu)$ be a measure space, let $g:X\to [0,+\infty)$ be a measurable function and  $u\ge 0$.
Let also $D$ be a measurable set such that $[g>u]\subset D\subset [g\ge u]$.
Then for every measurable set $K$ such that $\mu(D)=\mu(K)$ we have that $\int\limits_K gd\mu\le \int\limits_D gd\mu$.
\end{lemma}
\begin{proof}[\bf Proof.]
Let $\nu$ denote the indefinite integral of $g$ with respect to $\mu$, i.e. the measure defined by the formula
$\nu(A)=\int\limits_Ag d\mu $ for all $A\in \mc{A}$.

We set $V_1=[g>u]$ and $V_2=[g\ge u]$.
We have that
$\nu(K)=\nu(K\cap V_1)+\nu(K\setminus V_1)= \nu(V_1)-\nu(V_1\setminus K)+\nu(K\setminus V_1)$ and taking into account
 that $V_1\subset D\subset V_2$ we get that
 \begin{equation}\label{eqn1}
 \nu(K)-\nu(D)=-\nu(D\setminus V_1)-\nu(V_1\setminus K)+\nu(K\setminus V_1)
 \end{equation}

Since for every $x\in D\setminus V_1$ we have that $x\in V_2$ and hence $g(x)\ge u$ and since $V_1\subset D$ we get that
\begin{equation}\label{eqn2}
  -\nu(D\setminus V_1)\le -u\mu(D\setminus V_1)=-u\mu(D)+u\mu(V_1)
 \end{equation}

 For every $x\in V_1\setminus K$ we have that $g(x)>u$, thus
 \begin{equation}\label{eqn3}
  -\nu(  V_1 \setminus K)\le -u\mu( V_1 \setminus K)=-u\mu(V_1)+u\mu(K\cap V_1)
 \end{equation}

 Finally for every $x\in K\setminus V_1$ we have that $g(x)\le u$ thus
  \begin{equation}\label{eqn4}
  \nu( K\setminus V_1 )\le u\mu( K\setminus V_1 )=u\mu(K)-u\mu(K\cap V_1)
 \end{equation}

 Therefore from \eqref{eqn1}, \eqref{eqn2}, \eqref{eqn3}, \eqref{eqn4} we get that
 \[ \nu(K)-\nu(D)\le u\mu(K)-u\mu(D)=0\]
 i.e.  $\int\limits_K g d\mu\le \int\limits_Dg d\mu$.
 \end{proof}

\begin{proof}[\bf Proof of Theorem \ref{T2}.]

For the case $k\le \frac{f}{M}$ it is obvious that for every $\phi\in \mc{C}_{X,\mc{T}}(M,f)$ and $K\in\mc{A}$ with $\mu(K)=k$
it holds that $\int\limits_K\mc{M}_{\mc{T}}\phi d\mu\le M\mu(K)=kM$.

We  examine next the case where  $\frac{f}{M}< k\le  1$. Let $\phi\in\mc{C}_{X,\mc{T}}(M,f)$.
The functions $g_1,g_2:X\to[0,+\infty]$ defined as $g_1(t)=\mu([\mc{M}_{\mc{T}}\phi >t])$ and
 $g_2(t)=\mu([\mc{M}_{\mc{T}}\phi \ge t])$ are decreasing, $g_1$ is right continuous, $g_2$ is left continuous,
 $g_1(t)\le g_2(t)$ for all $t$, $\lim\limits_{t\to+\infty}g_2(t)=0$ and $g_2(0)=1$.
 By Lemma \ref{lem1} there exists $u\ge 0$ such that $g_1(u)\le k\le g_2(u)$, i.e.
 $\mu([\mc{M}_{\mc{T}}\phi >u])\le k\le \mu([\mc{M}_{\mc{T}}\phi \ge u])$.
 The fact that the probability measure space $(X,\mc{A},\mu)$ is  non-atomic yields the existence of
 a $D\in\mc{A}$ with  $[\mc{M}_{\mc{T}}\phi >u]\subset D\subset [\mc{M}_{\mc{T}}\phi \ge u]$ such that
 $\mu(D)=k$.  Lemma \ref{lem2} implies that for every $K\in\mc{A}$ with $\mu(D)=\mu(K)$ the inequality
 $\int\limits_K \mc{M}_{\mc{T}}\phi d\mu \le  \int\limits_D \mc{M}_{\mc{T}}\phi d\mu$ holds.

 We set $V_1=[\mc{M}_{\mc{T}}\phi>u]$ and $V_2=[\mc{M}_{\mc{T}}\phi \ge u]$.
 Since $V_1\subset D\subset V_2$ and thus $\mu(V_1)\le \mu(D)\le \mu(V_2)$, there exists unique $s\in [0,1]$
 such that $k=\mu(D)= s\mu(V_1)+(1-s) \mu(V_2)$ hence $\mu(D\setminus V_1)= (1-s)\mu(V_2\setminus V_1)$.

\begin{claim}
\[\int\limits_D \mtf d\mu =s\int\limits_{V_1} \mtf d\mu+(1-s)\int\limits_{V_2} \mtf d\mu\]
\end{claim}
\begin{proof}[\bf Proof.]
\begin{eqnarray*}
\int\limits_D \mtf d\mu  & =& \int\limits_{V_1} \mtf d\mu + \int\limits_{D\setminus V_1}\mtf d\mu\\
                        & =&    \int\limits_{V_1} \mtf d\mu + u\mu(D\setminus V_1) \\
                            & =&  \int\limits_{V_1} \mtf d\mu    +u(1-s)\mu(V_2\setminus V_1)\\
                         & =&       s \int\limits_{V_1} \mtf d\mu  +(1-s) \int\limits_{V_1} \mtf d\mu +(1-s)\int\limits_{V_2\setminus V_1}\mtf d\mu\\
                         &=& s\int\limits_{V_1} \mtf d\mu+(1-s)\int\limits_{V_2} \mtf d\mu.
 \end{eqnarray*}
\end{proof}

The set $V_1=[\mtf>u]$ can be written as $V_1=\bigcup\limits_{i\in A}I_i$ where $(I_i)_{i\in A}$ is a pairwise almost disjoint
 family of sets in $\mc{T}$, such that for each $i\in I$ the set $I_i$ is a maximal set of the family
 $ \{I\in \mc{T}:\; \frac{1}{\mu(I)}\int\limits_I\phi d \mu>u\}$.

We may also assume that the set $V_2=[\mtf \ge u]$ can be written
$V_2=\bigcup\limits_{j\in B}J_j$ where $(J_j)_{j\in B}$ is a pairwise almost disjoint
 family of sets in $\mc{T}$, such that for each $j\in J$ the set $J_j$ is a maximal set of the family
 $ \{J\in \mc{T}:\; \frac{1}{\mu(J)}\int\limits_J \phi d \mu   \ge  u\}$.

 This is possible because we may assume   that $\phi$ satisfies the following additional property:

 \begin{equation}\label{123}
  \forall x\in X \;\; \exists I_x\in\mc{T}:\; \mc{M}_{\mc{T}}\phi(x)=\frac{1}{\mu(I_x)}\int_{I_x} \phi d \mu.
  \end{equation}
 This is due to the following:

 \begin{nnremark}
 Let  $\phi:X\to \R^+$ be a measurable function such that $\int_X\phi d \mu =f$ and $\|\phi\|_{\infty}=M$.
 If we define the sequence \seq{\phi}{n}  of functions as follows
\[  \phi_n=\sum\limits_{I\in \mc{T}_{(n)}}\Av_I(\phi)\chi_I  \]
(where $\chi_I$ denotes the characteristic function of $I$),   then $\int_X\phi_n d\mu=f$, $\|\phi_n\|_{\infty}\le M$
and each $\phi_n$  satisfies \eqref{123}.
 Moreover, $(\mc{M}_{\mc{T}}\phi_n)_{n\in\N}$ is an increasing sequence of functions that converges pointwise to $\mc{M}_{\mc{T}}\phi$.
 This implies that $\int_B \mc{M}_{\mc{T}}\phi_n \to  \int_B \mc{M}_{\mc{T}}\phi$ for every measurable subset $B$ of $X$.
 \end{nnremark}

For each $i\in A$ we set $\mc{A}_i=\{I\in\mc{A}:\; I\subset I_i\}$,   $\mc{T}_i=\{I\in\mc{T}:\; I\subset I_i\}$
 and we define $\mu_i(I)=\frac{1}{\mu(I_i)}\mu(I)$ for every $I\in \mc{A}_i$.
 Then $(I_i,\mc{A}_i,\mu_i)$ is a non-atomic probability measure space and $\mc{T}_i$ is a tree in $\mc{A}_i$.
 Denoting by $\phi_i$ the restriction of $\phi$ on $I_i$, by $M_i$ the essential supremum of $\phi_i$ and
 setting $f_i(\phi)=\int\limits \phi_i d \mu_i=\frac{1}{\mu(I_i)}\int\limits_{I_i}\phi d \mu$ we observe that
  $0\le f_i(\phi)\le M_i\le M$ and $\phi_i\in  \mc{C}_{I_i,\mc{T}_i}(M_i,f_i(\phi),M_{2,i})$ for some $M_{2,i}$ with
  $0\le M_{2,i}\le f_i(\phi)$.
  In  the case where  $M_{2,i}< f_i(\phi)$,  Theorem \ref{T1}  yields
 \[\int\limits_{I_i}\mc{M}_{\mc{T}_i}\phi_i d\mu_i \le f_i(\phi)+\big(f_i(\phi)-M_{2,i}\big)   \log\big(\frac{M_i-M_{2,i}}{f_i(\phi)-M_{2,i}}\big)\]
 which implies that
 \[\int\limits_{I_i}\mc{M}_{\mc{T}_i}\phi_i d\mu_i \le f_i(\phi)+f_i(\phi)   \log\big(\frac{M_i}{f_i(\phi)}\big)\]
 while in the case $M_{2,i}=f_i(\phi)$  the last inequality is obvious.
 The fact that each $I_i$ is a maximal set  of the family
 $ \{I\in \mc{T}:\; \frac{1}{\mu(I)}\int\limits_I\phi d \mu>u\}$ imlpies that $\mc{M}_{\mc{T}_i}\phi_i(x)=\mc{M}_{\mc{T}}\phi(x)$ for every
  $x\in I_i$.
  Thus we get that
 \[\frac{1}{\mu(I_i)}\int\limits_{I_i}\mc{M}_{\mc{T}}\phi d\mu \le f_i(\phi)+f_i(\phi) \log(\frac{M_i}{f_i(\phi)})\;\; \mbox{ for every }i\in A. \]

By defining in a similar way $f_j(\phi)$ and $M_j$ for every $j\in B$, we get that
 \[\frac{1}{\mu(J_j)}\int\limits_{J_j}\mc{M}_{\mc{T}}\phi d\mu \le f_j(\phi)+f_j(\phi) \log(\frac{M_j}{f_j(\phi)})\;\; \mbox{ for every }j\in B. \]

From the inequalities above we get that
\begin{eqnarray*}
\int\limits_{V_1}\mc{M}_{\mc{T}}\phi d\mu & = & \sum\limits_{i\in A} \int\limits_{I_i}\mc{M}_{\mc{T}}\phi d\mu\\
        & \le &   \sum\limits_{i\in A} \mu(I_i) \Big( f_i(\phi)+f_i(\phi) \log(\frac{M_i}{f_i(\phi)})  \Big).
\end{eqnarray*}
and
\begin{eqnarray*}
\int\limits_{V_2}\mc{M}_{\mc{T}}\phi d\mu & = & \sum\limits_{j\in B} \int\limits_{J_j}\mc{M}_{\mc{T}}\phi d\mu\\
        & \le &   \sum\limits_{j\in B} \mu(J_j) \Big( f_j(\phi)+f_j(\phi) \log(\frac{M_j}{f_j(\phi)})  \Big).
\end{eqnarray*}

Let now an arbitrary $K\in\mc{A}$ with $\mu(K)=k$.
We have that
\begin{eqnarray*}
\int\limits_K \mc{M}_{\mc{T}}\phi d\mu & \le  &   \int\limits_D \mc{M}_{\mc{T}}\phi d\mu \\
   & = &   s\int\limits_{V_1}\mc{M}_{\mc{T}}\phi d\mu+(1-s)\int\limits_{V_2}\mc{M}_{\mc{T}}\phi d\mu\\
 & \le  &   s \sum\limits_{i\in A} \mu(I_i) \Big( f_i(\phi)+f_i(\phi) \log(\frac{M_i}{f_i(\phi)})  \Big)\\
     &  &   +(1-s)  \sum\limits_{j\in B} \mu(J_j) \Big( f_j(\phi)+f_j(\phi) \log(\frac{M_j}{f_j(\phi)})  \Big)\\
     & = &   s  \sum\limits_{i\in A} \mu(I_i)f_i(\phi)+(1-s) \sum\limits_{j\in B} \mu(J_j)f_j(\phi)\\
     &  &   +  s\sum\limits_{i\in A} \mu(I_i)f_i(\phi) \log(\frac{M_i}{f_i(\phi)})+
 (1-s) \sum\limits_{j\in B} \mu(J_j)  f_j(\phi) \log(\frac{M_j}{f_j(\phi)}).
\end{eqnarray*}
 Therefore
 \begin{eqnarray*}
\int\limits_K \mc{M}_{\mc{T}}\phi d\mu & \le  &   s \int\limits_{V_1}\phi d\mu+(1-s)\int\limits_{V_2}\phi d\mu\\
     &  &   +  s\sum\limits_{i\in A} \mu(I_i)f_i(\phi) \log(\frac{M_i}{f_i(\phi)})+
  (1-s) \sum\limits_{j\in B} \mu(J_j)  f_j(\phi) \log(\frac{M_j}{f_j(\phi)}).
\end{eqnarray*}
We observe that since $V_1\subset V_2$ we have that
\begin{equation}\label{V12}
  s \int\limits_{V_1}\phi d\mu+(1-s)\int\limits_{V_2}\phi d\mu\le  \int\limits_{V_2}\phi d\mu\le \int\limits_X \phi d\mu= f.
  \end{equation}
Also, since $M_i\le M$ for each $i\in A$ we get
\[\sum\limits_{i\in A} \mu(I_i)f_i(\phi) \log(\frac{M_i}{f_i(\phi)})\le
 \mu(V_1)\sum\limits_{i\in A}\frac{ \mu(I_i)}{\mu(V_1)}f_i(\phi)\log(\frac{M}{f_i(\phi)}).\]

Since $\sum\limits_{i\in A}\frac{\mu(I_i)}{\mu(V_1)}=1$ and the function $h_0$ defined as $h_0(t)=t\log(\frac{M}{t})$ is concave
we obtain
\begin{eqnarray*}
 \sum\limits_{i\in A}\frac{ \mu(I_i)}{\mu(V_1)}f_i(\phi)\log(\frac{M}{f_i(\phi)})\\
= \sum\limits_{i\in A}\frac{ \mu(I_i)}{\mu(V_1)} h_0(f_i(\phi))\\
\le   h_0\Big(\sum\limits_{i\in A}\frac{ \mu(I_i)}{\mu(V_1)}f_i(\phi)\Big)\\
= h_0\Big(\frac{1}{\mu(V_1)}\int\limits_{V_1}\phi d \mu\Big).
\end{eqnarray*}
Thus \[\sum\limits_{i\in A} \mu(I_i)f_i(\phi)\log(\frac{M_i}{f_i(\phi)})  \le \mu(V_1)h_0\Big(\frac{1}{\mu(V_1)}\int\limits_{V_1}\phi d \mu\Big).\]
Similarly it is shown that
\[    \sum\limits_{j\in B} \mu(J_j)f_j(\phi)\log(\frac{M_j}{f_j(\phi)})  \le \mu(V_2)h_0\Big(\frac{1}{\mu(V_2)}\int\limits_{V_2}\phi d \mu\Big).   \]

Therefore
\begin{eqnarray*}
 s\sum\limits_{i\in A} \mu(I_i)f_i(\phi)\log(\frac{M_i}{f_i(\phi)}) +(1-s) \sum\limits_{j\in B} \mu(J_j)f_j(\phi)\log(\frac{M_j}{f_j(\phi)}) \\
 \le s\mu(V_1)h_0\Big(\frac{1}{\mu(V_1)}\int\limits_{V_1}\phi d \mu\Big)+(1-s)\mu(V_2)h_0\Big(\frac{1}{\mu(V_2)}\int\limits_{V_2}\phi d \mu\Big) .
 \end{eqnarray*}

Taking into account that the number $s$ has been seleceted such that $s\mu(V_1)+(1-s)\mu(V_2)=\mu(D)=k$ and using again
the fact that the function $h_0$ is concave we get that the quantity above is
\begin{eqnarray*}
=k\Big(\frac{s\mu(V_1)}{k}h_0\Big(\frac{\int\limits_{V_1}\phi d \mu }{\mu(V_1)}\Big)+\frac{(1-s)\mu(V_2)}{k}h_0\Big(\frac{\int\limits_{V_2}\phi d \mu}{\mu(V_2)}\Big)\Big)\\
\le k h_0\Big( \frac{s}{k}\int\limits_{V_1}\phi d\mu + \frac{(1-s)}{k}\int\limits_{V_2}\phi d \mu \Big).
\end{eqnarray*}
We set $\Gamma=\frac{s}{k}\int\limits_{V_1}\phi d\mu + \frac{(1-s)}{k}\int\limits_{V_2}\phi d \mu $.
Until this point we have shown that
\[ \int\limits_K  \mc{M}_{\mc{T}}\phi d \mu \le  k \Gamma+k h_0(\Gamma)=k\Gamma+k\Gamma \log(\frac{Mk}{k\Gamma}).\]
We consider the function $h_1(t)=t+t\log(\frac{Mk}{t})$, $t\in (0,Mk]$.
The function $h_1$ is increasing, since $h_1'(t)=\log(\frac{Mk}{t})>0$ for every $t\in (0,Mk)$.
Since from \eqref{V12} we have that $k\Gamma\le f$ we get that
$h_1(k\Gamma)\le h_1(f)$  (remember that we are studying the case where $\frac{f}{M}<k\le 1$), thus
$k\Gamma+k\Gamma\log(\frac{Mk}{k\Gamma})\le f+f\log(\frac{Mk}{f})$.

Therefore
\[  \int\limits_K \mtf d\mu \le f+f\log(\frac{Mk}{f}).\]

Up to this point we have shown that the left hand side of \eqref{ineq1} in the statement of Theorem \ref{T2} is less than or equal to the right side.
It remains to show that equality holds. We will again discern the cases $0<k\le\frac{f}{M}$ and $\frac{f}{M}< k\le 1$.

Let's first consider the case $0<k\le\frac{f}{M}$.  Using Lemma \ref{L2} we select a family $(I_i)_{i\in A}$ of almost pairwise disjoint
 sets in $\mc{T}$ such that $\sum\limits_{i\in A} \mu(I_i)=k$.
 We consider the function $\phi: X\to\R$ defined as
 \[\phi(x)=\left\{ \begin{array} {l@{\quad} l}
  M & \mbox{ if }\;\;x\in  \bigcup\limits_{i\in A}I_i \\[4mm]
 \frac{f-Mk}{1-k} &   \mbox{ if }\;\;x\in X\setminus     \bigcup\limits_{i\in A}I_i   \end{array}\right.  \]

The function $\phi$ is measurable, while since $k\le \frac{f}{M}$ we have $\phi\ge 0$.
Also, since $f\le M$ we have $\frac{f-Mk}{1-k}\le M$ thus $\|\phi\|_{\infty}=M$.
Finally
\begin{eqnarray*}
 \int\limits_X\phi d \mu  & = & M\mu(\bigcup\limits_{i\in A}I_i)+\frac{f-Mk}{1-k}\mu(X\setminus     \bigcup\limits_{i\in A}I_i)\\
 & = &  Mk+\frac{f-Mk}{1-k}(1-k)= f
\end{eqnarray*}
thus $\phi\in \mc{C}_{X,\mc{T}}(M,f)$.

Since
\[     \int\limits_{\bigcup\limits_{i\in A}I_i}\mtf d \mu =Mk  \]
the proof in the case  where $0<k\le \frac{f}{M}$ is complete.

We treat now the case $\frac{f}{M}\le k\le 1$.
In this case it is enough to show that   for every $\delta >0$ there exists
 a $K\in \mc{A}$ with $\mu(K)=k$ and a $\phi\in \mc{C}_{X,\mc{T}}(M,f)$ such that
\[    \int\limits_K \mtf d\mu\ge (1-\delta) \big(f+f\log(\frac{Mk}{f})\big) . \]
 As in the first case, using again Lemma \ref{L2}, we select a family $(I_i)_{i\in A}$  of almost pairwise disjoint
 sets in $\mc{T}$ such that $\sum\limits_{i\in A} \mu(I_i)=k$.
 We set $K=\bigcup\limits_{i\in A}I_i$, $f'=\frac{f}{k}$ (notice that $f'\le M$) and for every $i\in A$
 we set $\mc{A}_i=\{I\in \mc{A} : \;I\subset I_i\}$,  $\mc{T}_i=\{I\in \mc{T}: \;I\subset I_i\}$ and
 we define $\mu_i(I)=\frac{1}{\mu(I_i)}\mu(I)$ for every $I\in \mc{A}_i$.  For each $i\in A$, Theorem \ref{T1}
 yields the
 existence of a $\phi_i\in \mc{C}_{I_i,\mc{T}_{i}}(M,f',0) $ such that
 $\int\mc{M}_{\mc{T}_{i}}\phi_i d\mu_i>(1-\delta)(f'+f'\log(\frac{M}{f'}))$.

 We consider  $\phi:X\to\R$ with $\phi|I_i=\phi_i$ for each $i\in I$ and $\phi|(X\setminus \bigcup\limits_{i\in A}I_i)=0$.
Then $\phi\ge 0$,  $\|\phi\|_{\infty}=M$ and
\[\int_{X} \phi d\mu =\sum\limits_{i\in A}\int\limits_{I_i}\phi d\mu= \sum\limits_{i\in A}\mu(I_i)f'=k\frac{f}{k}=f\]
hence $\phi\in \mc{C}_{X,\mc{T}}(M,f) $.

Finally we have that
\begin{eqnarray*}
\int\limits_K\mtf d\mu & = & \sum\limits_{i\in A}\int\limits_{I_i}\mtf d\mu\\
                         & \ge  & \sum\limits_{i\in A}\int\limits_{I_i}\mc{M}_{\mc{T}_{i}}\phi_i d\mu\\
                         & =  & \sum\limits_{i\in A}\mu(I_i) \int\limits_{I_i} \mc{M}_{ \mc{T}_{i}}\phi_i d\mu_i\\
                         & \ge &\sum\limits_{i\in A}\mu(I_i) (1-\delta)(f'+f'\log(\frac{M}{f'}))  \\
                         & = & k(1-\delta) (\frac{f}{k}+\frac{f}{k}\log(\frac{Mk}{f})\\
                         &=&  (1-\delta)(f+f\log(\frac{Mk}{f})).
\end{eqnarray*}
      Letting $\delta\to 0^+$ we obtain our result.
 \end{proof}


\begin{remark} \label{rf}
In the case that $f<M$ we may add the condition $\essinf_X(\phi)=0$ in the statement of Theorem \ref{T2} without affecting the value of the supremum.
\end{remark}
\begin{proof}[\bf Proof.] We distinguish the following three cases

 {\bf Case 1.}\;\;   $k=1$.   \\
 This case is obviously a special case of Theorem \ref{T1} for $M_2=0$.

  {\bf Case 2.}\;\;   $\frac{f}{M}\le k<1$.\\
  In this case the   extremal function $\phi$ that we constructed in the proof of Theorem \ref{T2}
  satisfies the desired condition.

    {\bf Case 3.}\;\;  $0<k<\frac{f}{M}$.\\
    In this case  the extremal  function $\phi$ that we constructed in the proof of Theorem \ref{T2} does not satisfy the condition $\essinf_X(\phi)=0$.
    In order to overcome this difficulty, for every sufficiently small $\e>0$ (namely $\e<1-\frac{f}{M}$) we construct a function $\phi_\e$ as follows.
    Keeping the notation that was used in the proof of Theorem \ref{T2}, for fixed $i_0\in A$, we choose a set $J_\e\in \mc{T}$
    such that $0<\mu(J_\e)<\e\mu(I_0)< 1-\frac{f}{M}$ and  $\mu(J_\e)<\frac{\e}{M}$.  We define  $\phi_\e$ by the formula

    \[\phi(x)=\left\{ \begin{array} {l@{\quad} l}
  M &  \mbox{if }x\in \bigcup\limits_{i\in A}I_i \setminus J_\e\\[3mm]
         0 &  \mbox{ if }\;\;x\in J_\e\\[3mm]
 \frac{f-M(k-\mu(J_\e))}{1-k} &   \mbox{ if }\;\;x\in X\setminus     \bigcup\limits_{i\in A}I_i.   \end{array}\right.  \]

 We notice that for this choice of $\e$ and $J_\e$ we have that $\frac{f-M(k-\mu(J_\e))}{1-k}<M$ and thus $\|\phi_\e\|_{\infty}=M$,
 while $\int \phi_\e d\mu =f$, $\essinf_X(\phi_\e)=0$ and
 \[
 \int\limits_{\bigcup\limits_{i\in A}I_i}\mc{M}_\mc{T}\phi_\e d \mu
 \ge  \int\limits_{\bigcup\limits_{i\in A}I_i\setminus J_{\e}}\mc{M}_\mc{T}\phi_\e d \mu
 =M(k-\mu(J_{\e}))\ge Mk-\e.
 \]
  and this completes the proof of the remark.
\end{proof}

An application of Theorem \ref{T2} is the following.

\begin{corollary}
For  $0\le M_2<f<M_1$ and $0\le k\le 1$
\begin{multline}
 \sup\Big\{ \int\limits_K  \mc{M}_{\mc{T}}\phi d\mu:\;\phi:X\to \R^+ \mbox{ is measurable},\;
  \int\limits_X \phi d\mu=f,\;\|\phi\|_{\infty}=M_1,  \\[1mm]
  \essinf_X(\phi)=M_2,  \; K \mbox{ measurable, }    \mu(K)=k      \Big\}                   \\[1mm]
=  \left\{ \begin{array} {l@{\quad} l}
        kM_1    & \mbox{ if }0<k\le \frac{f-M_2}{M_1-M_2}  \\[4mm]
   f-M_2(1-k) +(f-M_2)\log(\frac{(M_1-M_2)k}{f-M_2}) & \mbox{ if } \frac{f-M_2}{M_1-M_2}<k\le 1
   \end{array}\right.
  \end{multline}
\end{corollary}
\begin{proof}[\bf Proof.]
Fix $M_2,f,M_1$ and $k$ as in the statement, and let $A$ be the value of the supremum.
Then
\begin{multline*}
 A= kM_2+  \sup\Big\{ \int\limits_K  \mc{M}_{\mc{T}}(\phi-M_2) d\mu:\;\phi:X\to \R^+ \mbox{ is measurable},\\[1mm]
  \int\limits_X (\phi-M_2) d\mu=f-M_2,\;\;\|\phi-M_2\|_{\infty}=M_1-M_2,  \\[1mm]
  \essinf_X(\phi-M_2)=0,\;  \; K \mbox{ measurable, }    \mu(K)=k      \Big\}                   \\[1mm]
  =kM_2+\sup\Big\{ \int\limits_K  \mc{M}_{\mc{T}}\psi d\mu:\;\psi:X\to \R^+ \mbox{ is measurable},\;
  \int\limits_X \psi d\mu=f-M_2,\\[1mm]
  \|\psi\|_{\infty}=M_1-M_2, \;\;
  \essinf_X(\psi)=0,\;  \; K \mbox{ measurable, }    \mu(K)=k      \Big\}                   \\[1mm]
 \end{multline*}
 The latest supremum is evaluated using Remark \ref{rf}.
  In the case that  $0<k\le \frac{f-M_2}{M_1-M_2}$ we get that
\[
  A=kM_2+ k(M_1-M_2)  =  kM_1 \]
 while in the case that $0<k\le \frac{f-M_2}{M_1-M_2}$
 we get that
 \begin{eqnarray*}  A &  = &  kM_2+(f-M_2)+(f-M_2) \log(\frac{(M_1-M_2)k}{f-M_2})\\
    & = &  f-M_2(1-k)+(f-M_2) \log(\frac{(M_1-M_2)k}{f-M_2})
    \end{eqnarray*}
 and this completes the proof.
\end{proof}


\begin{thebibliography}{9}	


\bibitem{B1}
    D. L. Burkholder,
    \emph{Martingales and Fourier analysis in Banach spaces,}
    Probability and analysis (Varenna 1985), 61--108, Lecture Notes in Math., 61-108, {\bf 1206}, Springer, Berlin, 1986.	




\bibitem{B2}
    D. L. Burkholder,
    \emph{Explorations in martingale theory and its applications,}
    {\'E}cole d'{\'E}t{\'e} de Probabilit{\'e}s de Saint-Flour XIX—1989,  1-66, Lecture Notes in Math., {\bf 1464}, Springer, Berlin,   1991.


\bibitem{M1}
	A. D. Melas,
	\emph{The Bellman functions of dyadic-like maximal operators and related inequalities},
	Adv. Math. {\bf 192} (2005), no. 2,
	310-340.

	

\bibitem {M2} A. D. Melas.
    \emph{Sharp general local estimates for dyadic-like maximal operators and related Bellman functions,}
    Adv. in Math. {\bf 220} (2009), no. 2, 367-426.


 \bibitem {M3} A. D. Melas.
\emph{Dyadic-like maximal operators on $ L \log L$  functions},
 J. Funct. Anal. {\bf 257} (2009), no. 6, 1631-1654.


    \bibitem{NM1}
	E. N. Nikolidakis, A. D. Melas,
	\emph{A sharp integral rearrangement inequality for the dyadic maximal operator and applications},
	Appl. and  Comput. Harmon. Anal., {\bf 38} (2015), no. 2,   242-261.
	

	
	
	
 \bibitem{N2}
	E. N. Nikolidakis,
	\emph{Extremal problems related to maximal dyadic-like operators},
	J. Math. Anal. Appl, {\bf 369} (2010), no.1, 377-385.
	
	
	
	 \bibitem{N3}
	E. N. Nikolidakis,
	\emph{The Bellman function of the dyadic maximal operator in connection with the Dyadic Carleson Imbedding Theorem and related inequalities},
	 arXiv: 1905.08091.


\bibitem{S} E. M. Stein, \emph{Note on the class $L\log L$},   Studia Math. {\bf 32} (1969), 305-310.

 \bibitem{Wa}
      G. Wang,
      \emph{Sharp maximal inequalities for conditionally symmetric martingales and Brownian motion,}
       Proceedings of the American Mathematical Society {\bf 112} (1991), no. 2, 579-586.

\bibitem{Wi}
N. Wiener,  \emph{The ergodic theorem},  Duke Math. J. {\bf 5} (1939), no. 1, 1-18.
	
\end{thebibliography}
\end{document}